\documentclass[12pt,reqno]{amsart}

\usepackage{color}

\usepackage{amsmath,amsfonts,amsthm,amssymb,amsxtra}

\newtheorem{theorem}{Theorem}
\newtheorem{proposition}[theorem]{Proposition}
\newtheorem{lemma}[theorem]{Lemma}
\newtheorem{corollary}[theorem]{Corollary}

\theoremstyle{definition}

\theoremstyle{remark}

\newtheorem{remark}[theorem]{Remark}

\renewcommand{\epsilon}{\varepsilon}

\renewcommand{\phi}{\varphi}
\newcommand{\R}{\mathbb{R}}

\newcommand\eps\varepsilon
\newcommand\cd{\mathcal{C}_{\lambda,d}}

\newcommand{\rd}{\mathrm{d}}
\newcommand{\dx}{\rd x}
\newcommand{\dy}{\rd y}
\newcommand{\dt}{\rd t}
\newcommand{\cF}{{\mathcal F}}

\title{Stability estimate for the Lane-Emden inequality}

\author[E. Carlen]{Eric Carlen}
\address{Department of Mathematics, Rutgers University, 110 Frelinghuysen Rd., Piscataway NJ 08854-8019, USA}
\email{carlen@math.rutgers.edu}

\author[M. Lewin]{Mathieu Lewin}
\address{CNRS \& CEREMADE, Universit\'e Paris-Dauphine, PSL University, 75016 Paris, France}
\email{mathieu.lewin@math.cnrs.fr}

\author[E.H. Lieb]{Elliott H. Lieb}
\address{Departments of Mathematics and Physics, Jadwin Hall, Princeton University, Washington Rd., Princeton, NJ 08544, USA}
\email{lieb@princeton.edu}

\author[R. Seiringer]{Robert Seiringer}
\address{IST Austria (Institute of Science and Technology Austria), Am Campus 1, 3400 Klosterneuburg, Austria}
\email{robert.seiringer@ist.ac.at}

\date{\today.}

\begin{document}

\begin{abstract}
The  Lane-Emden inequality controls $\iint_{\R^{2d}}\rho(x)\rho(y)|x-y|^{-\lambda}\,\dx\,\dy$ in terms of the $L^1$ and $L^p$ norms of $\rho$. We provide a remainder estimate for this inequality in terms of a suitable distance of $\rho$ to the manifold of optimizers.

\bigskip

\noindent \sl \copyright~2024 by the authors. This paper may be reproduced, in its entirety, for non-commercial purposes.
\end{abstract}

\maketitle

\section{Introduction and main result}

\subsection{Lane-Emden inequality}
We consider the  inequality
\begin{equation}
\boxed{\iint_{\R^{2d}} \frac{\rho(x)\rho(y)}{|x-y|^\lambda} \dx\, \dy \leq a(\lambda,p,d) \left(\int_{\R^d}\rho(x)\,\dx\right)^{2-\frac {p\lambda}{d(p-1)}} \left(\int_{\R^d}\rho(x)^p\,\dx\right)^{\frac {\lambda}{d(p-1)}}}
 \label{eq:Lane-Emden}
\end{equation}
for every $0\leq \rho\in (L^1\cap L^p)(\R^d)$, whose validity follows from the Hölder and Hardy-Littlewood-Sobolev (HLS) inequalities~\cite{LieLos-01}. The conditions on $\lambda$ and $p$ are
\begin{equation}
0<\lambda<d,\qquad p_c:=\frac{2}{2-\lambda/d}<p\leq\infty.
\label{eq:cond_lambda_p}
\end{equation}
When $p=+\infty$ it is understood that the right side of~\eqref{eq:Lane-Emden} reads $\|\rho\|_1^{2-\lambda/d}\|\rho\|_\infty^{\lambda/d}$.
We call
\begin{equation}\label{LEp}
a = a(\lambda,p,d) := \sup \left\{\iint_{\R^{2d}} \frac{\rho(x)\rho(y)}{|x-y|^\lambda} \dx\, \dy \, : \, \rho\geq0,\ \|\rho\|_1 = 1 = \|\rho\|_p \right\}
\end{equation}
the best constant in~\eqref{eq:Lane-Emden}. In the Coulomb case $d=3$, $p=4/3$, $\lambda=1$  the optimizer of \eqref{eq:Lane-Emden} solves the Lane-Emden equation~\cite{Lane-70,Emden-07,Chandrasekhar-39}. For this reason we call~\eqref{eq:Lane-Emden} the \emph{Lane-Emden inequality}. 

By rearrangement, it is well known that optimizers of~\eqref{eq:Lane-Emden} are radial-decreasing after suitable translation~\cite{Lieb-77,LieOxf-80}. Uniqueness (up to the trivial symmetries) was investigated in \cite{LieOxf-80,CarHofMaiVol-18,ChaGonHuaMaiVol-20,CalCarHof-21}. The Lane-Emden case $d=3$, $\lambda=1$ and $p=4/3$ is contained in \cite[App.~A]{LieOxf-80} (see also~\cite[App.~A]{LieYau-87}). Using ODE techniques, that proof extends to $d>3$ for $\lambda = d-2$. The case $d-2 < \lambda < d$ and  $p\leq 2$ is contained in  \cite{ChaGonHuaMaiVol-20}. Moreover, uniqueness is shown in~\cite{CarHofMaiVol-18} for $d=1$, $0<\lambda<1$ and $p>1+\lambda$ and in~\cite{CalCarHof-21} for $d-2\leq \lambda < d$, $d\geq 2$ and $p\geq 1+\lambda/d$. The case $p\geq 2$ is also investigated in \cite{DelYanYao-20}.

In this paper we provide a refined inequality containing a remainder term, in the spirit of the refinement of Sobolev's inequality~\cite{BreLie-85,BiaEgn-91,DolEstFigFraLos-23_ppt} and several other recent results in the same direction. Our main result is the following.

\begin{theorem}[Stability for the  Lane-Emden inequality]\label{thm:main}
Assume that
\begin{equation}
p_c=\frac{2}{2-\lambda/d}< p\leq 2, \qquad \lambda > 0,\qquad d-2\leq \lambda < d,
\label{eq:assumption_p_lambda}
\end{equation}
and let $a=a(\lambda,p,d)$ denote the optimal Lane-Emden constant in \eqref{LEp}. 
Then there exists a constant $c = c(\lambda,p,d) > 0$ such that
\begin{equation}\label{main:eq}
\int_{\R^d}\rho(x)^p\,\dx \geq \left( a^{-1} \iint_{\R^{2d}} \frac{\rho(x)\rho(y)}{|x-y|^\lambda} \dx\, \dy \right)^\frac{d(p-1)} {\lambda}+ c  \inf_\ell  \int_{\R^d}\left(\rho(x)^{ \frac p2} - \ell(x)^{ \frac p2}\right)^2\dx
\end{equation}
for all non-negative $\rho \in L^p(\R^d)$ with $\int_{\R^d}\rho= 1$, where the infimum is over all optimizers of the Lane-Emden inequality~\eqref{eq:Lane-Emden} with mass $\int_{\R^d}\ell= 1$.
\end{theorem}

For our range of parameters~\eqref{eq:assumption_p_lambda}, uniqueness up to the trivial symmetries is proved in~\cite{ChaGonHuaMaiVol-20}. We thus know that the infimum in~\eqref{main:eq} is over functions of the form $\ell_{\kappa,y}(x) = \kappa^{d} \ell_{1,0}(\kappa (x-y))$ for $\kappa>0$ and $y\in \R^d$, where $\ell_{1,0}$ is the unique radial-decreasing Lane-Emden optimizer satisfying $\|\ell_{1,0}\|_1=\|\ell_{1,0}\|_p=1$. Let us emphasize that the bound~\eqref{main:eq} only involves the distance to the Lane-Emden optimizers \emph{having the same mass as $\rho$}. This is stronger than if we were minimizing over the mass as well. A natural question is whether one can, in addition, restrict the infimum to the $\ell$'s of the same $L^p$ norm as $\rho$. Our proof does not seem to allow it. Variations over $\kappa$ yield an additional orthogonality condition that plays an important role in our argument. Fixing $\|\rho\|_p$ instead of $\|\rho\|_1$ seems to be possible with our proof method, however. 

Note that the considered parameter range covers the usual Lane-Emden case $\lambda=d-2= d(p-1)$ (i.e., $p = 2 - 2/d$) for $d\geq 3$. The generalization to $p>2$ or $\lambda<d-2$ remains an open problem.

We follow a rather standard strategy for the proof. We start by studying the spectrum of the Hessian at an optimizer. The main result of this first step is formulated in Proposition~\ref{main:prop}, and Theorem~\ref{thm:scattering} below plays an essential role here. We then extend this result to prove local stability, which is somewhat tricky due to the various norms involved. Finally we get a global bound by resorting to a compactness argument from~\cite{Lions-84,BiaEgn-91}. Unfortunately, this last step of the method does not provide any explicit value for the stability constant~$c$ in~\eqref{main:eq}.

One novelty in our argument concerns the way we are able to get the non-degeneracy of the Hessian. We will explain below that the proof can be reduced to showing that a certain operator of the form $(-\Delta)^s+V(x)$ for suitable $0<s\leq 1$ has a trivial kernel in the space of radial functions going to zero at infinity.  The classical approach to handle this problem is to analyze the number of zeroes of the corresponding radial solution~\cite{Kwong-89,McLeod-93,Tao-06,Frank-13,FraLen-13,FraLenSil-16}. We argue differently and use instead a novel scattering-type argument which, we think, is of independent interest. It goes as follows. First we exhibit an $s$-wave zero-energy scattering solution $u$, that is, a radial solution of $\big((-\Delta)^s+V(x)\big)u=0$ converging to a non-trivial constant at infinity. Then we use that the existence of this solution implies the {\em non-existence} of a radial solution tending to 0 at infinity for the same equation, hence the triviality of the kernel. The proof of this implication in the local case is rather elementary, but the one for $0<s<1$ is not so easy. The precise statement is the following.

\begin{theorem}[No radial decaying solution if there is a scattering solution]\label{thm:scattering}
Let $0<s\leq1$ and $d\geq1$. Assume that $V:\R^d\mapsto \R$ is a radial, bounded, non-decreasing function that tends to zero at infinity. Let $f$ be a bounded radial solution of the equation
\begin{equation}\label{main:eq_scattering}
(-\Delta)^s f + V f = \tau V 
\end{equation}
for some $\tau \in \R$, with $\lim_{|x|\to\infty} f(x) = 0$. If $f(0) = \tau$, then $\tau=0$ and $f \equiv 0$. 

As a consequence, if there exists a solution $f$ as above with $\tau\neq0$, then there cannot exist a non-trivial bounded radial solution $g$ to $(-\Delta)^s g + V g =0$ tending to 0 at infinity.
\end{theorem}

We think here that $u=\tau-f$ is the scattering solution. Hence the theorem states that $s$-wave zero-energy scattering solutions cannot vanish at the origin. The last part of the statement is because if we had a non-trivial solution $g$ tending to 0 at infinity, then $g(0)\neq0$ by the theorem. But then $h=f+(\tau-f(0))g/g(0)$ would satisfy $h(0)=\tau$, a contradiction. 

The proof of Theorem~\ref{thm:scattering} is provided in Appendix~\ref{appendix}. It follows closely the work in~\cite{FraLen-13,FraLenSil-16}, where Theorem~\ref{thm:scattering} was proved in the case of $\tau=0$ under additional regularity assumptions on $V$. Since these regularity assumptions are not fulfilled in our case, we cannot directly apply the result in~\cite{FraLen-13,FraLenSil-16}, but we shall still closely follow the strategy developed there.

In our application, $s=(d-\lambda)/2$, and the restriction $s\leq 1$ in Theorem~\ref{thm:scattering} is the origin of the restriction $\lambda \geq d-2$ in Theorem~\ref{thm:main}. If Theorem~\ref{thm:scattering} can be proved for larger $s$, this would also immediately extend the validity of Theorem~\ref{thm:main}. Also the condition $p\leq 2$ enters in Theorem~\ref{thm:scattering} since only under this restriction $V$ is monotone increasing in our application. The condition $p\leq 2$ enters in other places of the proof as well, however.

\begin{remark}[Another proof of uniqueness]
The non-degeneracy of the linearized problem, which is proved in this paper, can also be used to recover the uniqueness result from~\cite{ChaGonHuaMaiVol-20}, by following a similar argument as in~\cite{FraLen-13,FraLenSil-16}. The idea is to follow the branch of any positive solution with Morse index one so as to reach the local case $\lambda=d-2$, where uniqueness is guaranteed by an ODE argument.
\end{remark}

\subsection{Application of the stability bound}

A functional  that is closely related to the Lane-Emden inequality has been studied recently  in connection with certain ``diffusion-aggregation'' type evolution equations. The functional is 
\begin{equation}\label{AGDIF2}
\cF(\rho) := \frac{1}{p-1}\int_{\R^d}\rho(x)^p{\rm d}x - \frac{\chi}{2} \iint_{\R^{2d}}\frac{\rho(x)\rho(y)}{|x-y|^\lambda}{\rm d}x\,{\rm d}y\ 
\end{equation}
where $\chi$ is a positive constant, $p>1$ and  
$0 < \lambda< d$. The free energy~\eqref{AGDIF2} naturally occurs in some  diffusion-aggregation equations specified below, which dissipate $\cF$ while conserving the mass, and hence the evolution drives $\rho$ toward lower values of $\cF$.  In this context, it is thus natural to assume that $\rho$ has a fixed total mass $M := \int_{\R^d}\rho(x)\,{\rm d}x$. The first question this  raises is when $\cF$ has a minimum under this constraint. When there are minimizers, the next question is how close $\rho$ must be to one of the minimizers given that $\cF(\rho)$ is close to its minimum value.  Our main theorem gives an answer to this second question. To explain this we briefly recall some known facts \cite{LieOxf-80,CarHofMaiVol-18,ChaGonHuaMaiVol-20,CalCarHof-21} concerning the existence of minimizers of $\cF$. 

Let us denote by 
\begin{equation}
    F_M(\lambda,p,d,\chi) =\inf\left\{ \cF(\rho) \ :\  0\leq \rho\in (L^1\cap L^p)(\R^d), \ \int_{\R^d}\rho(x){\rm d}x = 
M\right\}
    \label{eq:min_cF}
\end{equation}
the minimal value of $\cF$ (which could be $-\infty$). Defining  $\rho_a(x) = a^{-d}\rho(x/a)$ for $a>0$, we find that 
\begin{equation}\label{AGDIF2Z}
\cF(\rho_a) := a^{-d(p-1)}\frac{1}{p-1}\int_{\R^d}\rho(x)^p{\rm d}x - a^{-\lambda}\frac{\chi}{2} \iint_{\R^{2d}}\frac{\rho(x)\rho(y)}{|x-y|^\lambda}{\rm d}x\,{\rm d}y\ .
\end{equation}
Evidently, $\lim_{a\to \infty} \cF(\rho_a) = 0$ and therefore $F_M(\lambda,p,d,\chi)\leq0$.

When $\lambda > d(p-1)$, then  $\lim_{a\to 0^+} \cF(\rho_a) = -\infty$, so that the functional is unbounded below in this case, i.e., 
$F_M(\lambda,p,d,\chi)=-\infty$ for $p<1+\frac\lambda{d}$. 
However, if $d(p-1) > \lambda$, $\cF$ is bounded below as a consequence of the Lane-Emden inequality, and $F_M(\lambda,p,d,\chi)$ is finite. 

In the critical case $\lambda = d(p-1)$,  $\cF(\rho_a) = a^{-\lambda}\cF(\rho)$, and the Lane-Emden inequality~\eqref{eq:Lane-Emden} implies 
$$F_M(\lambda,p,d,\chi)=\begin{cases}
-\infty&\text{for $M>M_c(d,\lambda,\chi)$} \\
0&\text{for $M\leq M_c(d,\lambda,\chi)$}\\
\end{cases}$$
with
\begin{equation}\label{CRMASS}
M_c(d,\lambda,\chi) := \left(\frac{2d}{\chi\lambda\, a(\lambda,1+\lambda/d,d)}\right)^{\frac{d}{d-\lambda}}
\end{equation}
the critical mass. For $d= 3$ and $\lambda =1$, the critical value is $p = \frac43$. The sharp Lane-Emden constant $a(1,\frac43,3)$ and all of the 
optimizers were computed in \cite{LieOxf-80}. These results give the value $M_c(3,1,\chi)$. 

Next, we return to the case $p>1+\lambda/d$ where $\cF$ is bounded below for all values of the mass $M$. Let $\alpha := \frac {\lambda}{d(p-1)} < 1$. Applying the Lane-Emden inequality \eqref{eq:Lane-Emden} and optimizing over $\int \rho^p$ at fixed $\int \rho = M$ gives the value
\begin{equation}\label{eq:value_F}
F_M(\lambda,p,d,\chi) = \frac d \lambda (\alpha-1)  \left(\frac{\chi\lambda}{2d}a(\lambda,p,d) M ^{2- p\alpha}\right)^{\frac{1}{1-\alpha}}.
\end{equation}
We have $\cF(\rho)=F_M(\lambda,p,d,\chi)$ if and only if $\rho$ is a Lane-Emden optimizer satisfying in addition
\begin{equation}\label{AGDIF7}
\int_{\R^d}\rho(x)^p\,\dx = \left(\frac{\chi\lambda}{2d} a(\lambda,p,d) M ^{2- p \alpha}\right)^{\frac{1}{1-\alpha}}=:P_M(\lambda,p,d,\chi).
\end{equation}
Thus not every Lane-Emden minimizer $\rho$ with mass $M$ is a minimizer of $\cF$.  The functional $\cF$ is not scale invariant, as  displayed in \eqref{AGDIF2Z}, and in fact it has a preferred scale. The minimizers of $\cF$ with mass $M$ are precisely the Lane-Emden minimizers satisfying \eqref{AGDIF7}.

Our main result in Theorem~\ref{thm:main} can be used to obtain a remainder estimate for the inequality $\cF\geq F_M(\lambda,p,d,\chi)$, that is, quantifying how close a $\rho$ has to be to the manifold of Lane-Emden minimizers when $\cF(\rho)$ is close to its minimal value $F_M(\lambda,p,d,\chi)$. For simplicity, we take $M=1$ (the general case follows by scaling) and let $\mathcal{L}_\chi$ denote the set of all Lane-Emden optimizers satisfying \eqref{AGDIF7} (for $M=1$). 

\begin{corollary}\label{FCOR} Let $2\geq p > 1 + \frac{\lambda}{d}$ for $\lambda>0$ and $d-2< \lambda<d $. Let $\chi>0$. Then there exists a constant $c = c(\lambda,p,d) > 0 $ (independent of $\chi$) such that the functional $\cF$ in \eqref{AGDIF2} satisfies 
\begin{equation}\label{eq:cor}
\cF(\rho) - F_1(\lambda,p,d,\chi) \geq  c  \min_{\ell\in \mathcal{L}_\chi}  \int_{\R^d}\left(\rho(x)^{ \frac p2} - \ell(x)^{ \frac p2}\right)^2\dx
\end{equation}
for any $0\leq \rho\in (L^1\cap L^p)(\R^d)$ such that $\int_{\R^d}\rho(x)\,{\rm d} x=M=1$, where $F_1$ is defined in \eqref{eq:value_F}. 
\end{corollary}

The minimum in \eqref{eq:cor} is over translations only, and it is easy to see that it is attained.

\begin{proof}
    For $\alpha<1$, the function $g_\alpha(x) := x - x^\alpha/\alpha + 1/\alpha -1 $ is non-negative for $x\geq 0$, and vanishes if and only if $x=1$. Let us use the simpler notation $P:=P_1(\lambda,p,d,\chi)$ for the value \eqref{AGDIF7} that $\int_{\R^d}\rho^p$ must have for minimizers.  
    With the aid of $g_\alpha$ we can rewrite $\mathcal{F}$ as 
\begin{multline*}
\mathcal{F}(\rho) = F_1(\lambda,p,d,\chi) + \frac{\alpha \chi a}2 P^\alpha g_\alpha\left( P^{-1}\int \rho^p \right) \\ 
 + \frac{\chi a}2 \left( \int \rho^p\right)^\alpha -\frac \chi 2 \iint \frac{\rho(x)\rho(y)}{|x-y|^\lambda}{\rm d}x \, {\rm d} y .
\end{multline*}
Applying the inequality \eqref{main:eq} as well as the simple bound $A^\alpha - B^\alpha \geq \alpha(A-B)A^{\alpha-1}$ for $A\geq B\geq 0$  yields the lower bound
\begin{multline}
\cF(\rho) - F_1(\lambda,p,d,\chi) \geq 
 \frac{\alpha \chi a}2 P^\alpha g_\alpha\left( P^{-1}\int \rho^p\right)
\\+
\frac {\alpha\chi a c}2 \left( \int \rho^p\right)^{\alpha-1} \inf_\ell  \int_{\R^d}\left(\rho(x)^{ \frac p2} - \ell(x)^{ \frac p2}\right)^2\dx
. \label{tocF}
\end{multline}
Note that the infimum here is over {\em all} Lane-Emden optimizers, not necessarily satisfying \eqref{AGDIF7}. Our goal will be to restrict the minimization to the ones satisfying this additional constraint, using the first term involving $g_\alpha$. 

Let us distinguish two cases. If $\int \rho^p > 2P$, we can use the fact that $g_\alpha(x) \geq C_\alpha x $ for $x \geq 2$ and a suitable constant $C_\alpha>0$ to conclude that 
\begin{align*}
\cF(\rho) - F_1(\lambda,p,d,\chi) &\geq 
 \frac{\alpha d }\lambda   C_\alpha\int \rho^p 
\geq \frac{\alpha d }\lambda   C_\alpha
 \min_{\ell \in \mathcal{L}_\chi}  \int_{\R^d}\left(\rho(x)^{ \frac p2} - \ell(x)^{ \frac p2}\right)^2\dx
.
\end{align*}
We can thus focus on the case $\int \rho^p \leq 2P$. Given a Lane-Emden optimizer $\ell$, let $\ell_\chi \in\mathcal{L_\chi}$ denote the corresponding rescaled Lane-Emden optimizer which has $\int \ell_\chi = 1$ and $\int \ell_\chi^p = P$ (and the same center as $\ell$). For instance, if $\ell$ is radial about the origin then $\ell_\chi(x)=\kappa^d\ell(\kappa x)$ with 
$$\kappa:=\left(P^{-1}\int_{\R^d}\ell^p\right)^{-\frac{1}{d(p-1)}}\geq 2^{-\frac{1}{d(p-1)}}.$$
Then we use 
\begin{equation}\label{toc1}
\| \rho^{p/2} - \ell_\chi^{p/2} \|_2 \leq \| \rho^{p/2} -\ell^{p/2}\|_2 + \| \ell^{p/2} - \ell_\chi^{p/2}\|_2
\end{equation}
and claim that the last term can be estimated by
\begin{equation}
\| \ell^{p/2} - \ell_\chi^{p/2}\|_2 \leq C \left| \| \ell^{p/2}\|_2 - \| \ell_\chi^{p/2}\|_2\right|=C\sqrt{P}\left|\kappa^{\frac{d(1-p)}2}-1\right|,
\label{eq:upper_L_P}    
\end{equation}
for a suitable constant $C>0$. Indeed, assuming again the center of $\ell$ is at the origin, we have
\begin{align*}
\| \ell^{p/2} - \ell_\chi^{p/2}\|_2^2&=P\left(1+\kappa^{d(1-p)}\right)-2\kappa^{\frac{dp}{2}}\int_{\R^d}\ell(x)^{\frac{p}2}\ell(\kappa x)^{\frac{p}2}  
\end{align*}
which is $C^{1,1}$ with respect to  $\kappa$, since $\ell^{p/2}$ is a Lipschitz function ($\ell^{p-1}$ is, and $p\leq 2$). Moreover, we can bound $g_\alpha(x) \geq C_\alpha' (\sqrt{x}-1)^2$ for suitable $C_\alpha'>0$, and hence
\begin{equation}\label{toc2}
\| \ell^{p/2} - \ell_\chi^{p/2}\|_2^2 \leq C^2 \| \ell_\chi^{p/2}\|_2^2  (C_\alpha')^{-1} g_\alpha\left( P^{-1}\int \rho^p\right).
\end{equation}
Combining \eqref{toc1} and \eqref{toc2} with \eqref{tocF} yields the claimed result. 
\end{proof}

Corollary~\ref{FCOR} has applications to a family of ``diffusion-aggregation'' type evolution equations for mass densities $\rho$ on $\R^d$. These equations take the form
\begin{equation}\label{AGDIF1}
\frac{\partial}{\partial t}\rho  = \Delta \rho^p - \chi \nabla\cdot \rho (\nabla \rho\ast|x|^{-\lambda}),
\end{equation}
where $\chi$ is a constant, $p>1$ and $0 < \lambda< d$.   Notice that $M = \int_{\R^d}\rho(x)\,{\rm d}x$ is at least formally conserved by the evolution.  The  non-linear diffusion term competes against the aggregating effect of the attractive potential.  Since
$$\frac{\delta \cF}{\delta\rho}= \frac{p}{p-1}\rho^{p-1} - \chi\, \rho\ast|x|^{-\lambda},$$ 
the equation \eqref{AGDIF1} can be written as
\begin{equation}\label{AGDIF27}
\frac{\partial}{\partial t}\rho  = \nabla\cdot \left(\rho \nabla  \frac{\delta \cF}{\delta\rho}\right)\ ,
\end{equation}
from which it follows that along the flow described by the equation, 
$$ \frac{{\rm d}}{{\rm d}t}\cF(\rho)   = -\int_{\R^d} \rho \left|\nabla  \frac{\delta \cF}{\delta\rho}\right|^2{\rm d}x \leq 0\ .$$ 

There is at present no rate information on how fast $\cF(\rho(\cdot,t))$ decreases to its minimum value along the flow described by \eqref{AGDIF1}. It is not even known in general that the flow arrives at a global minimizer of $\cF$,  although such information has been obtained for the  closely related Keller-Segel equation in $d=2$ \cite{BlaCarCar-12,CarFig-13}. As explained in this section, granted a bound on the rate of decay of
$\cF(\rho(\cdot,t)) - F_1(\lambda,p,d,\chi)$, 
the stability estimate proved here would give a quantitative estimate on how close $\rho(\cdot ,t) $ is to a Lane-Emden optimizer $\ell$. 

Finally, we remark that the restriction to $M=1$ in Corollary~\ref{FCOR} entails no loss of generality. For a solution $\rho(x,t)$ of \eqref{AGDIF1} that has mass $M$, define $\tilde \rho(x,t) = M^{-1}\rho(x,M^{1-p}t)$ and $\tilde{\chi} = M^{2-p}\chi$.  Then $\tilde\rho(x,t)$ solves \eqref{AGDIF1}  with $\chi$ replaced by $\tilde\chi$, and has unit mass.

\section{Proof of the main result}

Throughout the proof we use the simplified notation
$$
D(\rho_1,\rho_2) = \iint_{\R^{2d}} \frac{\rho_1(x)\rho_2(y)}{|x-y|^\lambda} \dx\, \dy.
$$

\subsection{Euler-Lagrange equation}

Let $\ell$ be a maximizer for \eqref{LEp}, with $\|\ell\|_1 = 1 = \|\ell\|_p$.
We expand
$$
\frac{D(\ell + \eps g, \ell + \eps g)}{\|\ell + \eps g\|_p^{\frac {p\lambda}{d(p-1)}}} = a  + 2 \eps  D(\ell, g) - {\frac {p\lambda}{d(p-1)}} \eps a \int \ell^{p-1} g + o(\eps)
$$
and conclude that
$$
\ell* |x|^{-\lambda} = {\frac {p\lambda a}{2d(p-1)}} \ell^{p-1} + \mu
$$
on the support of $\ell$, and $\leq$ otherwise. Multiplying this by $\ell$ and integrating, we find
\begin{equation}\label{def:mu}
\mu = a \left( 1 - {\frac {p\lambda}{2d(p-1)}} \right) > 0.
\end{equation}
We conclude that the Euler-Lagrange equation reads
\begin{equation}\label{ELe}
\boxed{ {\frac {p \lambda a}{2d(p-1)}}  \ell^{p-1} = \left[ \ell* |x|^{-\lambda}  - \mu\right]_+}
\end{equation}
 and that $\ell$ necessarily has compact support.

\subsection{Computation of the Hessian}
As usual, the stability in Theorem~\ref{thm:main} will follow from a non-degeneracy property of the Hessian. We thus start by computing it.

Let $\ell$ be any Lane-Emden maximizer. Without loss of generality, we can assume that $\ell$ is radial decreasing about the origin. For $g$ with $\int \sqrt{\ell} g = 0$, we  have
\begin{align*}
&\frac{D( (\sqrt{\ell} + \eps g)^2, (\sqrt{\ell} + \eps g)^2 ) }{\|\sqrt\ell + \eps g\|_{2p}^{\frac {2p\lambda}{d(p-1)}}\|\sqrt\ell + \eps g\|_{2}^{4-\frac {2p\lambda}{d(p-1)}} }  = a
\\ &  +  \eps^2 \left(  4 D(\sqrt\ell g, \sqrt\ell g) - \frac{2 \lambda p a }{d} \int \ell^{p-1} g^2 - \frac{2\lambda p^2 a}{d(p-1)} \left( \frac \lambda{d(p-1)}-1\right) \left( \int \ell^{p-1/2} g \right)^2 \right)  \\ &+ o(\eps^2),
\end{align*}
which identifies the Hessian (at least on the support of $\ell$) as $Q H Q$, where $Q = 1 - | \sqrt\ell\rangle\langle\sqrt\ell |$ and
\begin{equation}\label{def:H}
H = 4 \sqrt{\ell} R \sqrt{\ell} -  \frac{2 \lambda p a}{d}  \ell^{p-1} - \frac{2\lambda p^2 a}{d(p-1)} \left( \frac \lambda{d(p-1)}-1\right) | \ell^{p-1/2} \rangle \langle \ell^{p-1/2}|.
\end{equation}
We introduced the notation $R$ for the convolution with $|x|^{-\lambda}$. Note that we can also write $R = \cd (-\Delta)^{(\lambda-d)/2}$ for a suitable constant $\cd >0$.
We shall abbreviate
\begin{equation}\label{def:alpha}
\alpha :=  \frac{2\lambda p^2 a}{d(p-1)} \left( \frac \lambda{d(p-1)}-1\right),
\end{equation}
which equals zero in the \lq\lq local\rq\rq\ case $\lambda = d(p-1)$.

The maximizing property of $\ell$ implies that $QHQ\leq 0$. In the following we want to prove that there exists a $\kappa>0$ such that
\begin{equation}\label{gap}
\langle g , H g\rangle \leq - \kappa \int g^2 \ell^{p-1}
\end{equation}
for all $g \perp \{ \sqrt\ell, \nabla \sqrt\ell , \varphi\}$ with $\varphi = \frac d2 \sqrt\ell + x\cdot \nabla \sqrt\ell$. We shall in fact prove \eqref{gap} for all $g \perp \{ \sqrt \ell, \ell^{p-1} \nabla \sqrt\ell , \ell^{p-1} \varphi\}$ which will be more natural and convenient below. 

We shall write
$$
H= \ell^{(p-1)/2} \left( A -  \frac{2 \lambda p  }{d}  a - \alpha | \ell^{p/2} \rangle\langle \ell^{p/2} | \right) \ell^{(p-1)/2}.
$$
Note that the operator
\begin{equation}
A = 4 \ell^{1-p/2} R \ell^{1-p/2}=4 \ell^{1-p/2} \big(|x|^{-\lambda}\ast\cdot\big) \ell^{1-p/2}
\label{eq:def_A}
\end{equation}
is clearly compact for $p\leq 2$. For $p=2$, we interpret $ \ell^{1-p/2}$ as the characteristic function of the support of $\ell$.

The main result of this section is the following non-degeneracy property of the operator $A$ in~\eqref{eq:def_A}.

\begin{proposition}[Non-degeneracy]\label{main:prop}
Assume that $p_c< p\leq 2$ and $d-2\leq \lambda < d$. Then there exists a constant $\kappa = \kappa(\lambda,p,d) > 0$ such that
\begin{equation}\label{d3}
\langle h ,A h \rangle -\alpha \langle h , \ell^{p/2}\rangle^2 \leq \left(  \frac {2p\lambda} d a  - \kappa\right) \|h\|^2 \quad \forall h \perp \{\ell^{1-p/2},\ell^{p/2-1} \nabla \ell, \ell^{(p-1)/2}\varphi\}.
\end{equation}
\end{proposition}

The rest of the section is devoted to the proof of this proposition.

\begin{proof}
Since $\ell$ is radial,  we can  distinguish the various angular momentum channels for the spectral analysis. Since the letter $\ell$ is already in use, we shall denote the angular momentum quantum number by $m$.

\medskip

\noindent {\bf Case $m \geq 2$}. Since $A$ is strictly decreasing in angular momentum, and $A\leq \frac {2\lambda p}d a$ for $m =1$, we have $A\leq\frac {2\lambda p}d a- \kappa$ for $m \geq 2$ for some $\kappa>0$ by compactness.

\medskip

\noindent {\bf Case $m = 1$}. Note that $A \ell^{p/2-1} \nabla \ell = \frac {2 p \lambda}d a  \ell^{p/2-1} \nabla \ell$ (as differentiation of the Euler-Lagrange equation \eqref{ELe} for $\ell$ shows). By Perron-Frobenius (and the fact that $\ell^{p/2} $ is a radial decreasing function), the largest eigenvalue is non-degenerate (apart from the trivial $d$-fold degeneracy), hence by compactness of $A$
$$
\langle h | A | h\rangle \leq \left( \frac {2 p \lambda}d a - \kappa\right) \| h \|^2 \quad \text{if $\langle h| \ell^{p/2-1} \nabla \ell\rangle =0$,}
$$
which proves the desired inequality.

\medskip

\noindent {\bf Case $m = 0$}. A simple calculation shows that
\begin{equation}\label{sw}
\left(A - \frac {2p\lambda} d a \right) \ell^{(p-1)/2} \varphi = p\lambda a\left( \frac{\lambda}{d(p-1)} - 1\right) \ell^{p/2} +  2\lambda \mu \ell^{1-p/2}
\end{equation}
with $\mu$ given in \eqref{def:mu}.
Since
$$
\langle \ell^{p/2} , \ell^{(p-1)/2} \varphi \rangle = \frac d 2 \left ( 1 - \frac 1p\right)
$$
we get
$$
\left(A - \frac {2p\lambda} d a - \alpha | \ell^{p/2} \rangle\langle \ell^{p/2} |\right) \ell^{(p-1)/2} \varphi =  2\lambda \mu \ell^{1-p/2}
$$
with $\alpha$ defined in \eqref{def:alpha}. Since $\varphi \perp \sqrt\ell$, this also implies that  $\langle \ell^{(p-1)/2} \varphi ,A \ell^{(p-1)/2} \varphi\rangle  - \alpha \langle \ell^{(p-1)/2} \varphi, \ell^{p/2} \rangle^2 = \frac {2p\lambda}d a\| \ell^{(p-1)/2} \varphi \|^2$.
Note also that, because of \eqref{ELe},
\begin{equation}\label{sw2}
\left(A - \frac {2p\lambda} d a \right) \ell^{p/2}  =
    {\frac {2 p \lambda a}{d}} \left( \frac 1{p-1} -1 \right) \ell^{p/2}  + 4 \mu \ell^{1-p/2}.
\end{equation}
Let $\tilde Q$ be the projection orthogonal to $\ell^{1-p/2}$. Assume that 
$$
\tilde Q \left( A - \frac {2p\lambda}d a - \alpha | \ell^{p/2} \rangle\langle \ell^{p/2} | \right) y =0
$$
for some non-zero $y\perp \{\ell^{1-p/2}, \ell^{(p-1)/2}\varphi\}$.
Then
$$
\left( A - \frac {2p\lambda}d a\right) y  = \alpha \langle \ell^{p/2} , y\rangle    \ell^{p/2} + \sigma \ell^{1-p/2}
$$
for some $\sigma$. {\em Assume} for the moment that $A$ does not have the eigenvalue $ \frac {2p\lambda}d a$. It follows from \eqref{sw} and \eqref{sw2}, and the non-vanishing of the corresponding $2\times 2$ determinant
$$
\det \left( \begin{array}{cc} p\lambda a\left( \frac{\lambda}{d(p-1)} - 1\right)  & 2\lambda \mu \\ {\frac {2 p \lambda a}{d}} \left( \frac 1{p-1} -1 \right) & 4 \mu \end{array} \right) = 4\mu p \lambda a \left( \frac \lambda d -1\right) < 0
$$
 that $y$ must be a linear combination of $\ell^{(p-1)/2}\varphi$ and $\ell^{p/2}$, and hence vanishes by the orthogonality assumption.
This proves the desired bound \eqref{d3} by compactness.

We are left with showing that $A$ does not have an eigenvalue $ \frac {2p\lambda}d a$. For this we use a novel scattering argument. We note that \eqref{sw} and \eqref{sw2} imply in combination that
\begin{equation}\label{sw3}
\left(A - \frac {2p\lambda} d a \right) f = 4\tau \ell^{1-p/2}
\end{equation}
with
$$
f=  {\frac {2}{d}} \left( \frac 1{p-1} -1 \right) \ell^{(p-1)/2}\varphi - \left( \frac{\lambda}{d(p-1)} - 1\right) \ell^{p/2}
$$
and
$$
\tau =  \mu \left(1  -  {\frac {\lambda  }{d}}\right)> 0.
$$
Using that $R = \cd (-\Delta)^{-s}$ with $s= (d-\lambda)/2$, we can rewrite \eqref{sw3}  as
\begin{equation}\label{psieq}
\left( \frac {2p\lambda a}{d \cd} (-\Delta)^s  +W\right) \psi = \tau W
\end{equation}
for $\psi  = R \ell^{1-p/2} f$ and $W=- 4 \ell^{2-p}$. Note that our assumptions imply that $0<s\leq 1$. 
This is where we can use Theorem~\ref{thm:scattering}. Recall that the latter states that the existence of a zero-energy scattering solution as in~\eqref{psieq} implies the absence of a radial solution vanishing at infinity.
By the Birman-Schwinger principle~\cite[Sec.~1.2.8]{FraLapWei-LT}, this in turn implies the absence of an eigenvalue $1$ of $(d/2p \lambda a) A$, as desired.
This completes the proof of Proposition~\ref{main:prop}.\end{proof}

\begin{remark}
It is a consequence of our proof that radial-decreasing solutions $V\geq0$ to the nonlinear equation
\begin{equation}
(-\Delta)^sV=(V-1)_+^{\frac1{p-1}}
\label{eq:NLS_LE}
\end{equation}
are non-degenerate, in the sense that the linearized operator satisfies
$$\ker\left((-\Delta)^s-\frac{(V-1)_+^{\frac{2-p}{p-1}}}{p-1}\right)={\rm span}\left\{\partial_{x_j}V,\ j=1,...,d\right\}.$$
The equation~\eqref{eq:NLS_LE} is similar, although not identical, to the fractional nonlinear Schr\"o\-dinger equation studied in~\cite{FraLen-13,FraLenSil-16}, where $(V-1)_+^{\frac1{p-1}}$ is replaced by $V^{\frac1{p-1}}-V$. When $s=1$ (that is, $\lambda=d-2$),~\eqref{eq:NLS_LE} is the Lane-Emden equation.
\end{remark}

\subsection{Stability}
Let $\ell \geq 0$ be the unique radial maximizer of the Lane-Emden problem \eqref{LEp}.
In the previous section we have shown that, as a consequence of Proposition~\ref{main:prop}, there exists a $\kappa>0$ such that
\begin{equation}\label{gap2}
\langle g , H g\rangle \leq - \kappa \int g^2 \ell^{p-1}
\end{equation}
for all  $g \perp \{ \sqrt \ell, \ell^{p-1} \nabla \sqrt\ell , \ell^{p-1} \varphi\}$, where the Hessian $H$ is defined in \eqref{def:H}.
We shall now argue that this implies a
stability estimate of the desired form \eqref{main:eq}. Equivalently, by uniqueness of Lane-Emden maximizers up to symmetries,
we have to show that for some $c>0$
\begin{equation}\label{tos}
 \|\rho\|_p^{p} \geq \left( a^{-1} D(\rho,\rho) \right)^\frac{d(p-1)} {\lambda}+ c  \inf_{\kappa,y}   \| \rho^{ \frac p2} - \ell_{\kappa,y}^{ \frac p2}\|_{2}^2
\end{equation}
for all $L^1$-normalized, non-negative $\rho$, where
$\ell_{\kappa,y}(x) = \kappa^{d} \ell(\kappa (x-y))$ for $\kappa>0$ and $y\in \R^d$.

Let us first assume that the infimum in \eqref{tos} is attained. This will be the case if $\inf_{\kappa,y}    \int ( \rho^{p/2} -\ell_{\kappa,y}^{p/2})^{2} < \int \rho^{p}$, since
$$
\lim_{y\to \infty}   \int \left( \rho^{p/2} -\ell_{\kappa,y}^{p/2} \right)^{2} = \int \rho^{p} + \kappa^{d(p-1)} \int \ell^{p},
$$
$$
\lim_{\kappa \to 0}   \int \left( \rho^{p/2} -\ell_{\kappa,y}^{p/2} \right)^{2}  = \int \rho^{p},
$$
while
$$
\lim_{\kappa\to\infty}   \int \left( \rho^{p/2} -\ell_{\kappa,y}^{p/2} \right)^{2} = +\infty.
$$
If $\inf_{\kappa,y}    \int ( \rho^{p/2} -\ell_{\kappa,y}^{p/2})^{2} < \int \rho^{p}$, the infimum is thus a minimum, and by the scale and translation invariance of the problem,
we can assume without loss of generality  that it is attained at $\kappa=1$, $y=0$.

Let us denote
$$
\delta = \rho^{p/2} - \ell^{p/2} \ , 
\quad X = \rho - \ell - \frac 2p \ell^{1-p/2} \delta.
$$
Note that $0\leq X \leq |\delta|^{2/p}$ under our assumption $p\leq 2$.
The fact that $\inf_{\kappa,y}    \int ( \rho^{p/2} -\ell_{\kappa,y}^{p/2})^{2}$ is minimized at $\kappa=1$, $y=0$, implies that $\delta \perp \{ \nabla \ell^{p/2}, d p \ell^{p/2} /2 + x\cdot \nabla \ell^{p/2}\}$. In other words,  $ \ell^{(1-p)/2} \delta \perp \{ \ell^{p-1} \nabla \sqrt\ell, \ell^{p-1}\varphi\} $.

We have
$$
\int \rho^p = 1 + 2 \int \ell^{p/2} \delta + \int \delta^2 
$$
and
\begin{align}\nonumber
D(\rho,\rho) & = a + 2  D(\ell, 2p^{-1} \ell^{1-p/2} \delta+X) + \frac 4 {p^2} D(\ell^{1-p/2} \delta,\ell^{1-p/2} \delta) \\ & \quad + 2 D (X, 2p^{-1} \ell^{1-p/2}\delta ) + D(X,X).\label{drr}
\end{align}
Note that the Lane-Emden equation \eqref{ELe} can equivalently be written as
\begin{equation}\label{ELeq2}
 {\frac {p \lambda a}{2d(p-1)}}  \ell^{p-1} =  \ell* |x|^{-\lambda}  - \mu + V_\ell
\end{equation}
with
$$
V_\ell(x) = \left[  \mu- \ell* |x|^{-\lambda}   \right]_+,
$$
which is a non-negative function supported on the complement of the support of $\ell$.
It follows from   \eqref{ELeq2} and $\int (\rho-\ell) = 0$  that
$$
2  D(\ell, 2p^{-1} \ell^{1-p/2} \delta+X) = \frac{2\lambda a}{d(p-1)} \int \ell^{p/2} \delta - 2  \int V_\ell \rho + \frac{p\lambda a}{d(p-1)} \int \ell^{p-1} X.
$$

We start with some \emph{a priori} estimates. All the terms on the right hand side of \eqref{drr}, except for the constant $a$, can be bounded in terms of either $\|\delta\|_2$ or  $\|X\|_{p_c}$ (by HLS, using also H\"older in the form $\| \ell^{1-p/2} \delta\|_{p_c} \leq \| \ell^{1-p/2} \|_{2/(1-\lambda/d)} \|\delta\|_2$). The latter can be bounded as
$$
\| X\|_{p_c} \leq \| X\|_1^{(p/p_c-1)/(p-1)} \| X\|_{p}^{p(1-1/p_c)/(p-1)} \leq \|\delta\|_2^{2(1-1/p_c)/(p-1)} \left( \int X\right)^{(p/p_c-1)/(p-1)}
$$
using H\"older and  $0\leq X \leq |\delta|^{2/p}$. Moreover,
$$
\int X = -\frac 2p \int \ell^{1-p/2} \delta \leq \frac 2 p \|\delta\|_2 \| \ell^{1-p/2}\|_2
$$
and hence all the terms are small as long as $\|\delta\|_2$ is small. In fact
$$
\|X\|_{p_c} \leq  \left( \frac 2 p \| \ell^{1-p/2}\|_2\right)^{(p/p_c-1)/(p-1)}   \|\delta\|_2^{(1+(p-2)/p_c)/(p-1)}
$$
and it will be important below that $(1+(p-2)/p_c)/(p-1)>1$ for $p<2$. For $p=2$, this bound will not be sufficient, however, and a separate argument will be needed to cover this case.

We conclude that, for small enough $\|\delta\|_2$,
\begin{align}\nonumber
& \int \rho^p - \left( a^{-1} D(\rho,\rho) \right)^{\frac {d(p-1)}\lambda} = 1 + 2 \int \ell^{p/2} \delta + \int \delta^2   \\ \nonumber &- \biggl( 1 +    \frac{2\lambda }{d(p-1)} \int \ell^{p/2} \delta - 2 a^{-1} \int V_\ell \rho + \frac{p\lambda }{d(p-1)} \int \ell^{p-1} X \\ \nonumber & \quad + \frac 4 {a p^2} D(\ell^{1-p/2} \delta,\ell^{1-p/2} \delta) + 2 a^{-1} D (X, 2p^{-1} \ell^{1-p/2}\delta ) + a^{-1} D(X,X) \biggl)^{\frac {d(p-1)}\lambda}
\\ & =   \int \delta^2    - p \int \ell^{p-1} X - \frac {d(p-1)}\lambda \frac 4 {a p^2} D(\ell^{1-p/2} \delta,\ell^{1-p/2} \delta)   \nonumber \\ & \quad -  2 \left(1 - \frac \lambda {d(p-1)}\right)\left(  \int \ell^{p/2} \delta \right)^2 - \mathcal{E}_1 \label{lastl}
\end{align}
with an error term $\mathcal{E}_1$ satisfying
$$
\mathcal{E}_1 \leq C \left(  \| \delta\|_2^3 + \|\delta\|_2 \|X\|_{p_c}  \right)-  \frac 1C  \int V_\ell \rho
$$
for suitable $C>0$. Note that
$$
\delta^2  -p \ell^{p-1}  X = \rho^{p} +  (p-1) \ell^p  - p\ell^{p-1} \rho = \delta^2 F_p (1+\delta/\ell^{p/2})
$$
where
$$
F_p(x) =  \frac{x^2 - 1  + p -p x^{2/p} }{(1-x)^2}.
$$
For $p\leq 2$, $F_p$ is an increasing function of $x\geq 0$. 
It satisfies $F_p(0) = p-1$, $F_p(1)=2(p-1)/p$ and $\lim_{x\to \infty} F_p(x) =1$.

In the following we shall assume that $p<2$.
Pick a $0<\omega<1$ and let $\Omega$ denote the set where $\delta  < - \omega \ell^{p/2}$, which is a subset of the support of $\ell$. Recall the HLS exponent $ p_c=2/(2-\lambda/d)$. By H\"older
$$
\| \ell^{1-p/2} \delta \chi_\Omega\|_{p_c}  \leq \| \delta\|_2 \| \ell^{1-p/2} \chi_\Omega\|_{2/(1-\lambda/d)} \leq \| \delta\|_2 \| \ell^{1-p/2} \chi_\Omega\|_{p/(1-p/2)} |\Omega|^{1/p_c - 1/p}.  
$$
Moreover,
$$
\|\delta\|_2^2 \geq \int_\Omega \delta^2 \geq \omega^2 \int_\Omega \ell^{p} = \omega^2  \| \ell^{1-p/2} \chi_\Omega\|_{p/(1-p/2)}^{p/(1-p/2)}
$$
and hence, in combination,
\begin{equation}\label{com1}
\| \ell^{1-p/2} \delta \chi_\Omega\|_{p_c}  \leq \| \delta\|_2^{2/p} \omega^{1-2/p}   |\Omega|^{1/p_c - 1/p}.
\end{equation}
This will allow us to replace $D (\ell^{1-p/2}\delta,  \ell^{1-p/2}\delta)$ by $D (\ell^{1-p/2}\delta \chi_{\Omega^c},  \ell^{1-p/2}\delta\chi_{\Omega^c})$, using HLS.

Let $P$ denote the orthogonal projection onto the $(d+2)$-dimensional space spanned by  $\{\ell^{1-p/2}, \ell^{(p-1)/2} \nabla \sqrt\ell, \ell^{(p-1)/2} \varphi\}$, and let  $P^\perp= 1-P$.
In order to be able to apply \eqref{gap2} to $\delta \chi_{\Omega^c}$, we want to first replace it  by $P^\perp \delta \chi_{\Omega^c}$. For this purpose, we estimate
\begin{equation}\label{com2a}
\| \ell^{1-p/2} P \delta \chi_{\Omega^c} \|_{p_c} \leq \| \ell^{1-p/2}\|_{\frac2{1-\lambda/d}} \| P \delta \chi_{\Omega^c} \|_{2} \leq\| \ell^{1-p/2}\|_{\frac2{1-\lambda/d}} \left(  \| P \delta \|_{2}+ \| P \delta \chi_{\Omega} \|_{2}\right).
\end{equation}
By construction, $\delta \perp \{ \ell^{(p-1)/2} \nabla \sqrt\ell, \ell^{(p-1)/2}\varphi\}$, hence
$$
\| P \delta \|_2  =  \|\ell^{1-p/2}\|_2^{-1} \left|  \int \ell^{1-p/2} \delta \right| =  \frac p2  \|\ell^{1-p/2}\|_2^{-1}   \int X.
$$

In the following, we shall need the fact that the functions in the range of $P$ are all bounded by (a constant times) $\ell^{1-p/2}$. This is equivalent to $\ell^{p-1}$ being a Lipschitz function, which is shown in~\cite[Prop.~2.12]{ChaGonHuaMaiVol-20}.
We can thus bound
\begin{equation}\label{com2b}
\| P \delta \chi_{\Omega} \|_{2} \leq C \| \ell^{1-p/2}\chi_\Omega\|_2 \|\delta\|_2 \leq C \omega^{1-2/p} \|\delta\|_2^{2/p}
\end{equation}
similarly as above.

The first two terms on the right side of \eqref{lastl} are bounded from below as
$$
\int\left( \delta^2  - p \ell^{p-1}  X \right) \geq  F_p(1- \omega) \int_{\Omega^c}  \delta^2 + F_p(0) \int_\Omega \delta^2.
$$
For the next two terms, we have, on the one hand,
\begin{align*}
& \frac {d(p-1)}\lambda \frac 4 {a p^2} D(\ell^{1-p/2} P^\perp \delta \chi_{\Omega^c},\ell^{1-p/2} P^\perp \delta \chi_{\Omega^c})  +  2 \left(1 - \frac \lambda {d(p-1)}\right)   \langle  \ell^{p/2} | P^\perp \delta \chi_{\Omega^c}\rangle^2
\\ & \leq  \left(  2\frac {p-1}{ p}   -  \frac {d(p-1)}{\lambda a p^2}\kappa \right) \| P^\perp \delta \chi_{\Omega^c} \|_2^2
 \end{align*}
because of \eqref{gap2}. On the other hand, the terms involving $P \delta \chi_{\Omega^c}$ and $\delta \chi_{\Omega}$, respectively, can be bounded using HLS with \eqref{com1} and \eqref{com2a}--\eqref{com2b}.
We conclude that
\begin{align}\nonumber
& \int \rho^p - \left( a^{-1} D(\rho,\rho) \right)^{\frac {d(p-1)}\lambda} \\ &\geq  \left( F_p(1- \omega)  -  F_p(1)  +  \frac {d(p-1)}{\lambda a p^2}\kappa \right)  \int_{\Omega^c}  \delta^2 + F_p(0) \int_\Omega \delta^2- \mathcal{E}_1- \mathcal{E}_2 \label{lastl2}
\end{align}
with
$$
\mathcal{E}_2 \leq C\left(  \omega^{2-4/p} \|\delta\|_2^{4/p} +  \omega^{1-2/p} \|\delta\|_2^{1+2/p}+ \|\delta\|_2 \int X 
 \right).
$$
Choosing $\omega$ small enough, this is of the desired form (for $p<2$), except for the error term involving $\int X$ in $\mathcal{E}_2$.
In order to bound it, it will be important that we have the negative term $ - \int V_\ell X =  - \int V_\ell \rho$ from $\mathcal{E}_1$ at our disposal.

As already noted above, $0\leq X \leq |\delta|^{2/p}$ for $p\leq 2$. Pick a ball $B$ that is strictly larger than the support of $\ell$. We bound
$$
\int X \leq \int_B |\delta|^{2/p} + \int_{B^c} \rho \leq \|\delta\|_2^{2/p} |B|^{1-1/p} + \int_{B^c} \rho.
$$
Moreover,
$$
\int V_\ell \rho \geq C_B \int_{B^c} \rho
$$
for a suitable constant $C_B$ depending on the radius of $B$. We thus conclude that if $\|\delta\|_2$ and $\int_{B^c} \rho$ are suitably small, then \eqref{tos} holds.

It is in fact enough to assume that $\|\delta\|_2$ is small enough. Due to the normalization condition $\int \rho = \int \ell = 1$ we have
$$
\int |\sqrt\rho - \sqrt\ell|^2 = 2 \int \sqrt\ell (\sqrt\ell - \sqrt\rho) \leq 2 \|\sqrt\ell -\sqrt\rho \|_{2p} \| \ell \|_{p/(2p-1)}^{1/2}.
$$
Note also that
$$
\int |\sqrt\ell - \sqrt\rho |^{2p} \leq \int | \ell^{p/2} - \rho^{p/2}|^2.
$$
In particular, $\int_{B^c}\rho $ is small if $\|\delta\|_2$ is. We conclude that as long as $\|\delta\|_2$ is small enough, \eqref{tos} holds. Note that this assumption on $\delta$ also implies $\|\delta\|_2^2 < \int \rho^p$, as we assumed in the beginning to guarantee that the infimum in \eqref{tos} is attained.

It remains to consider the case $p=2$. Since $F_2$ equals the constant function $1$,  there is no need for the cutoff $\omega$ in this case, and we can apply \eqref{gap2} directly to the function $P^\perp \delta$, with the result that
$$
 \int \rho^2 - \left( a^{-1} D(\rho,\rho) \right)^{\frac {d}\lambda}
\geq    \frac {d}{4 \lambda a}\kappa   \int  \delta^2 - \mathcal{E}_1-  C    \|\delta\|_2 \int X.
$$
As already mentioned, the bounds above on $\|X\|_{p_c}$ and $\int X$ are not good enough to conclude that the error terms are small compared to $\|\delta\|_2^2$. For $p=2$, we can use the fact that $X$ vanishes identically on the support of $\ell$, and equals $\rho = \delta$ outside the support. By splitting the integral into a ball slightly larger than the support, and its complement, we can bound
$$
\int X \leq   C \sqrt\epsilon \|\delta\|_2 +  \frac 1 \epsilon \int V_\ell \rho
$$
for any $\epsilon>0$,
where we used that $V$ increases linearly outside the support of $\ell$. Moreover, by H\"older
$$
\| X\|_{p_c} \leq \| X\|_2^{2(1-1/p_c)} \|X\|_1^{2/p_c-1}  \leq \|\delta\|_2^{2(1-1/p_c)} \left( \int X \right)^{2/p_c-1}.
$$
In combination, this proves the desired bound for an appropriate choice of $\epsilon$. That is, \eqref{tos} also holds in the case $p=2$ as long as $\|\delta\|_2$ is small enough.

The case when $\|\delta\|_2$ is not small is classical~\cite{BiaEgn-91,Frank-13,CarFraLie-14} and the argument goes by contradiction. Let us assume that there exists a sequence $\rho_n\geq0$ with $\int\rho_n=1=\int\rho_n^p$ such that
\begin{equation}
\inf_{\kappa,y} \int ( \rho_n^{p/2} -\ell_{\kappa,y}^{p/2})^{2} \geq \frac12 \int \rho_n^{p}=\frac12
 \label{eq:delta_large}
\end{equation}
and
\begin{equation}
\frac{ \|\rho_n\|_p^{p} - \left( a^{-1} D(\rho_n,\rho_n) \right)^\frac{d(p-1)} {\lambda}}{\inf_{\kappa,y}   \| \rho_n^{ \frac p2} - \ell_{\kappa,y}^{ \frac p2}\|_{2}^2}\to0.
\end{equation}
Since the denominator is bounded, this implies that $\rho_n$ is an optimizing sequence for the Lane-Emden inequality~\eqref{eq:Lane-Emden}. Such sequences are known to be compact in $L^p(\R^d)$ up to translations by~\cite[Thm. II.1, Cor. II.1]{Lions-84} and to converge after extraction to a Lane-Emden optimizer. This contradicts~\eqref{eq:delta_large} and concludes the proof of Theorem~\ref{thm:main}.\qed

\appendix
\section{Proof of Theorem~\ref{thm:scattering}}\label{appendix}
In this appendix we give the proof of Theorem~\ref{thm:scattering}, following closely~\cite{FraLen-13,FraLenSil-16}. After taking the real and imaginary parts we can assume that $f$ is real.

\subsection{Local case $s=1$}
When $s=1$ the equation can be written in radial coordinates as
$$-f''-\frac{d-1}{r}f'+Vf=\tau V.$$
For radial functions, we use everywhere the simplified notation $f(r):=f(re_1)$ with $e_1=(1,0,...,0)$. Next we consider the local Hamiltonian
\begin{equation}
H(r)=\frac{f'(r)^2}2-V(r)\frac{(\tau-f(r))^2}{2},
    \label{eq:local_Ham}
\end{equation}
which satisfies
$$H'(r)=-\frac{(d-1)}r f'(r)^2-V'(r)\frac{(\tau-f(r))^2}{2}\leq0.$$
Here $V'$ is understood as a non-negative measure over $\R_+$. Since $H$ is non-increasing and $f$ is radial, we must have
$$H(0)=-V(0)\frac{(\tau-f(0))^2}2\geq H(+\infty)=0.$$
Hence, if $f(0)=\tau$, we conclude that $H\equiv0$ and thus $f'\equiv0$. Since $f\to0$ at infinity, we deduce $f\equiv0$ and therefore $\tau=f(0)=0$, as was claimed.

 \begin{remark}
Using ODE techniques, one can show the validity of Theorem~\ref{thm:scattering} in the local case $s=1$ also without the monotonicity assumption on $V$. It remains an open problem whether such an extension is also possible for $s<1$.
 \end{remark}

\subsection{Non-local case $0<s<1$}
Here we employ the techniques and results in \cite{FraLenSil-16}. First, it follows from \cite[Prop.~2.9]{Silvestre-06} that $f \in C^{0,\alpha}(\R^d)$ for $\alpha < 2s\leq 1$, and even $f \in C^{1,\alpha}(\R^d)$ for $\alpha < 2s- 1$ in case $s>1/2$. In particular, $f \in C^{0,\alpha}(\R^d)$ for some $s< \alpha <1$.

Let now $u$ be the  $s$-harmonic extension of $f$ to $\R^d \times \R_+$, that is,
\begin{equation}\label{def:u}
u(x,t) = \int_{\R^d} P_s(x-y,t) f(y)\, \dy,
\end{equation}
where
$$
P_s(x,t) = c_{n,s} \frac{ t^{2s} }{ \left( t^2 + |x|^2 \right)^{n/2+s}}
$$
and $c_{n,s}>0$ is chosen such that $\int_{\R^d} P_s(x,t) \,\dx = 1$ for all $t>0$. As a first step we prove the following

\begin{lemma}The integrals $\int_0^\infty t^{1-2s} |\nabla_x u(x,t)|^2\dt$ and $\int_0^\infty t^{1-2s}  |\partial_t u(x,t)|^2 \dt$ are finite, continuous in $x$ and tend to 0 at infinity.
\end{lemma}

\begin{proof}
Note that 
\begin{equation}\label{eq:tu}
\partial_t u(x,t) = \int_{\R^d} \partial_t P_s(x-y,t) f(y) \,\dy
\end{equation}
and hence
$$
\left| \partial_t u(x,t) \right| \leq \|f \|_\infty  \int_{\R^d} \left| \partial_t P_s(y,t) \right|  \,\dy = C t^{-1} \|f\|_\infty.
$$
Moreover, we also have 
$$
\partial_t u(x,t) = \int_{\R^d} \partial_t P_s(x-y,t) \left( f(y)  - f(x) \right)\dy
$$
and hence
$$
\left| \partial_t u(x,t) \right| \leq \|f \|_\alpha  \int_{\R^d} \left| \partial_t P_s(y,t)\right| |y|^\alpha  \dy = C t^{-1+\alpha} \|f\|_\alpha,
$$
since $\alpha < 2s$, as explained before. Here we introduced as in~\cite[Sect.~4]{Stein-70}
$$
\|f\|_\alpha := \sup_{x,y\in\R^n} \frac{ |f(x)-f(y)|}{|x-y|^\alpha}.
$$
From these bounds, it readily follows that $\int_0^\infty t^{1-2s} | \partial_t u(x,t) |^2 \dt$ is uniformly bounded. By dominated convergence, it is continuous and goes to zero at infinity, since $\partial_t u$ does (for every fixed $t>0$, which follows from \eqref{eq:tu}). 

For $\nabla_x u(x,t)$ we can argue very similarly. We have 
\begin{equation}\label{eq:xu}
\nabla_x u(x,t) = \int_{\R^d} \nabla_x P_s(x-y,t) f(y) \,\dy
\end{equation}
and hence
$$
\left| \nabla_x u(x,t) \right| \leq \|f \|_\infty  \int_{\R^d} \left| \nabla_y P_s(y,t) \right|  \dy = C t^{-1} \|f\|_\infty.
$$
Moreover, 
$$
\nabla_x u(x,t) = \int_{\R^d} \left( \nabla_x P_s(x-y,t) \right) \left( f(y)  - f(x) \right) \dy
$$
and hence
$$
\left| \nabla_x u(x,t) \right| \leq \|f \|_\alpha  \int_{\R^d} \left| \nabla_y P_s(y,t)\right| |y|^\alpha   \dy = C t^{-1+\alpha} \|f\|_\alpha.
$$
Again we conclude that $\int_0^\infty t^{1-2s} | \nabla_x u(x,t) |^2 \dt$ is uniformly bounded, continuous and  goes to zero at infinity. Moreover, note that if $f$ is radial, then this expression vanishes at $x=0$, since $\nabla_x u(x,t)$ does for every $t>0$.   
\end{proof}

Next we define the function
\begin{equation}
 H(x) = d_s\int_0^\infty t^{1-2s} \left( |\nabla_x u(x,t)|^2 - |\partial_t u(x,t)|^2 \right) \dt  - V(x) |f(x)-\tau|^2,
 \label{eq:H_fractional}
\end{equation}
where $d_s = 2^{2s-1}\Gamma(s)/\Gamma(1-s)$.
This is the equivalent of the local Hamiltonian~\eqref{eq:local_Ham}. It is bounded, tends to 0 at infinity, and satisfies the following 

\begin{proposition}\label{prop:H_fractional_monotone}
Under the assumptions of Theorem~\ref{thm:scattering}, $x\mapsto H(x)$ in~\eqref{eq:H_fractional} is non-increasing in $r=|x|$. 
\end{proposition}

Proposition~\ref{prop:H_fractional_monotone} was proved in~\cite[Sec.~4.2]{FraLenSil-16} in the case  $\tau=0$ under stronger regularity assumptions on $f$ and $V$. These assumptions are not always fulfilled in our application, however. E.g., for $p=2$ in our application $V$ is a characteristic function. Hence we cannot directly apply the result in~\cite[Sec.~4.2]{FraLenSil-16} and develop a new argument that does not need this additional regularity.

\begin{proof}
We pick a non-negative $\eta \in C_c^\infty(\R_+)$  and define an averaged function
$$
H_{\rm av}(r) = \int_0^\infty H(\lambda)\, \eta(\lambda /r) \frac{\rd\lambda }{r}.
$$
Recall that $H(\lambda)$ means $H(\lambda e_1)$.
The function $H_{\rm av}$ is clearly differentiable. The analogue of \cite[Lemma~4.1]{FraLenSil-16} is the statement
\begin{equation}\label{eq:havp}
H_{\rm av}'(r) =  - \int_0^\infty \left( 2 d_s \frac{d-1}{\lambda}  \int_0^\infty t^{1-2s} |\partial_\lambda u(\lambda, t)|^2 dt + V'(\lambda) \left| f(\lambda) -\tau \right|^2 \right) \eta( \lambda/r) \frac{\rd\lambda}{r}.
\end{equation}
Here $V'$ has to be understood as a non-negative measure which is integrated against the mentioned continuous functions. In particular, $H_{\rm av}$ is non-increasing and, since this holds for any $\eta$, consequently also $H$ is non-increasing. 

In order to prove of \eqref{eq:havp}, we cannot proceed directly as in \cite[Lemma~4.1]{FraLenSil-16}, because of the lower regularity in our setting. Instead, we argue as follows. Let $j\in C_0^\infty(\R^n)$ be radial and non-negative, with $\int_{\R^d} j = 1$. For $\eps>0$, let 
$$
f_\eps(x) = \eps^{-d}  \int_{\R^d} j((x-y)/\eps) f(y)\, \dy
$$
and let $u_\eps$ be its $s$-harmonic extension. Let 
$$
H^\eps(x) = d_s\int_0^\infty t^{1-2s} \left( |\nabla_x u_\eps(x,t)|^2 - |\partial_t u_\eps(x,t)|^2 \right) \dt  - V_\eps(x) |f_\eps(x)-\tau|^2.
$$
One checks that $H^\eps(x)$ converges to $H(x)$ as $\eps\to 0$ for almost every $x$. By dominated convergence, we thus have
\begin{align*}
H_{\rm av}'(r) & =  - \int_0^\infty H(\lambda) \left(  \eta' (\lambda /r) \frac \lambda {r} + \eta(\lambda/r)  \right)  \frac{\rd\lambda }{r^2} 
\\ & =  - \lim_{\eps\to 0} \int_0^\infty H^\eps(\lambda) \left(  \eta' (\lambda /r) \frac \lambda {r} + \eta(\lambda/r)  \right)  \frac{\rd\lambda }{r^2}.
\end{align*}
Integration by parts thus yields
$$
H_{\rm av}'(r)   =   \lim_{\eps\to 0} \int_0^\infty \partial_\lambda H^\eps (\lambda) \eta (\lambda /r) \frac \lambda {r^2}\, \rd\lambda.
$$
To compute the derivative of $H^\eps$, we can now proceed as in \cite[Lemma~4.1]{FraLenSil-16}, with the result that
\begin{align*}
\partial_\lambda H^\eps (\lambda)&= - 2d_s  \frac{n-1}{\lambda} \int_0^\infty t^{1-2s} |\partial_\lambda u_\eps(\lambda,t)|^2 \dt  - V_\eps'(\lambda ) | f_\eps(\lambda) - \tau |^2 \\ & \quad - 2 \partial_\lambda f_\eps(\lambda) \Big( V_\eps(\lambda) f_\eps(\lambda) - ( V f)_\eps (\lambda) \Big).
\end{align*}
Now the first line converges to the desired expression (in the sense of measures) as $\eps\to 0$. We shall now argue that the second line converges to zero (when integrated against any bounded function of compact support), which completes the proof. 

For any $x\in\R^d$, we have
$$
V_\eps(x) f_\eps(x) - \left( V f \right)_\eps (x)  = \frac 12  \int_{\R^{2d}}  j^\eps(x-y) j^\eps(x-z) \left( V(y) - V(z) \right) \left( f(z) - f(y) \right) \dy\, \rd z,
$$
where we introduced $j^\eps(x)=\eps^{-d} j(x/\eps)$. Since $V$ is a radial and monotone function, 
$$
V(y) - V(z) = -\int_{|y|}^{|z|} \rd V
$$
for $|y|<|z|$, for some non-negative measure $\rd V$ on $\R_+$. Hence, using that 
$$
\nabla f_\eps(x) = \int_{\R^d} \nabla j^\eps(x-w) \left( f(w) - f(x) \right) \rd w , 
$$
we obtain
\begin{align*}
&\int | \nabla f_\eps| \left| V_\eps f_\eps - \left(Vf\right)_\eps \right| 
\\ & \leq \int_{\R^{2d}} \dx \, \rd w \int_{|y|<|z|} \dy\, \rd z \int_{|y|}^{|z|} \rd V \, j^\eps(x-y) j^\eps(x-z) |\nabla j^\eps(x-w)|\times\\
&\qquad\qquad\times \left| f(z) - f(y)\right| \left| f(w) - f(x) \right| 
\\ & \leq \|f\|_\alpha^2 \int_{\R^{2d}} \!\!\dx \, \rd w \int_{|y|<|z|} \!\!\dy\, \rd z \int_{|y|}^{|z|}\!\! \rd V \, j^\eps(x-y) j^\eps(x-z) |\nabla j^\eps(x-w)| \left| z -y\right|^\alpha \left| w - x \right|^\alpha 
\\ & = \eps^{\alpha -1} \| |\cdot|^\alpha \nabla j\|_1  \|f\|_\alpha^2   \int_{|y|<|z|}  \dy\, \rd z \int_{|y|}^{|z|}\!\! \rd V \, j^\eps * j^\eps(z-y)  \left| z -y\right|^\alpha 
\\ & = \eps^{\alpha -1} \| |\cdot|^\alpha \nabla j\|_1  \|f\|_\alpha^2  \int_{0}^\infty \rd V(t)  \int_{\R^d}  \rd z  \, j^\eps * j^\eps(z) |z|^\alpha \left| B_t(0) \cap B_t^{\rm c}(z) \right|.
\end{align*}
Since $\left| B_t(0) \cap B_t^{\rm c}(z) \right| \leq C |z| t^{n-1}$, the expression above is of order $\eps^{2\alpha}$ as long as  $\int_0^\infty \rd V(t)\, t^{n-1}$ is finite.
If we only integrate over a compact set, say a centered ball, $|x|$ stays bounded, and hence so do $|y|$ and $|z|$ because of the compact support of $j$. It is thus enough to integrate $dV(t) t^{n-1}$ over a compact interval, which is always finite (since $\int_0^\infty \rd V = \|V\|_\infty < \infty$). We have thus proved \eqref{eq:havp}. In particular, $H_{\rm av}$ is non-increasing, and since this holds for any choice of $\eta$, also $H$ is non-increasing. 
\end{proof}

With Proposition~\ref{prop:H_fractional_monotone} at hand, the end of the proof of Theorem~\ref{thm:scattering} goes as in the local case above. If $f(0) = \tau$, then
\begin{equation}\label{h0}
 H(0) = - d_s\int_0^\infty t^{1-2s} |\partial_t u(0,t)|^2\,  \dt  \leq 0
\end{equation}
and since $H$ is a decreasing function that vanishes as infinity, this implies $H\equiv0$. In particular, we conclude from \eqref{eq:havp} and \eqref{h0} that
\begin{equation}\label{u0t}
\int_0^\infty t^{1-2s} |\partial_t u(0,t)|^2  \,\dt = 0,
\end{equation}
that $f\equiv \tau$ on the support of $V'$, and also that 
$$
\int_0^\infty t^{1-2s} |\nabla_x u(x,t)|^2  \dt = 0
$$
for all $x\in \R^d$ for $d\geq 2$. For $d\geq 2$ one thus immediately concludes that $\nabla_x u\equiv 0$, hence $u\equiv0$ and $f\equiv0$ (since $f$ was assumed to vanish at infinity). For $d=1$, we can proceed exactly as in the proof of \cite[Thm.~2.1]{FraLenSil-16}, using \eqref{u0t}; we will not repeat the details here. 

Finally, if there exists a solution $f$ with $\tau\neq0$ then there cannot exist a non-trivial solution $g$ with $\tau=0$, tending to 0 at infinity. We have quickly explained the argument after the statement of Theorem~\ref{thm:scattering}. 

\section*{Acknowledgement}
We are grateful to Rupert Frank and Enno Lenzmann for helpful discussions.


\newcommand{\etalchar}[1]{$^{#1}$}

\end{document}